\documentclass[11pt]{amsart}
\usepackage[english]{babel}
\usepackage{amsfonts}
\usepackage[utf8]{inputenc}
\usepackage{amsthm}
\usepackage{ mathrsfs }
\usepackage{ amsmath }
\usepackage{mathrsfs}
\usepackage{enumerate}
\usepackage{enumitem}
\usepackage{cite}
\usepackage{bbm}
\usepackage{xcolor}
\usepackage{comment}
\usepackage{graphicx}
\usepackage{frontespizio}
\usepackage[makeroom]{cancel}
\usepackage[normalem]{ulem}
\usepackage{ amssymb }
\usepackage[top=2.8cm, bottom=2.8cm, left=3cm, right=3cm]{geometry}
\usepackage{hyperref}
\usepackage{mathtools}
\usepackage[T1]{fontenc}
\usepackage{lmodern}

\usepackage[toc,page]{appendix}
\usepackage{ esint }
\author{Matthew Hyde, Michele Villa and Ivan Yuri Violo}
\title[Ricci curvature bounded below and uniform rectifiability]{Ricci curvature bounded below \\ and \\ uniform rectifiability}


\newcommand{\eps}{\varepsilon}

\newcommand{\rr}{\mathbb{R}}
\newcommand{\nn}{\mathbb{N}}

\newcommand{\nchi}{{\raise.3ex\hbox{$\chi$}}}

\newcommand{\sfd}{{\sf d}}

\renewcommand{\phi}{\varphi}
\newcommand{\restr}[1]{\lower3pt\hbox{$|_{#1}$}}

\newcommand{\X}{{\rm X}}

\definecolor{mygray}{gray}{0.9}
\newcommand{\diam}{\text{diam}}

\newcommand{\mea}{\mathfrak{m}}
\newcommand{\mm}{\mathfrak{m}}

\renewcommand{\d}{{\mathrm d}}

\newcommand{\supp}{\mathop{\rm supp}\nolimits}

\newcommand{\Xdm}{(\X,\sfd,\mm)}
\newcommand{\RCD}{\mathrm{RCD}}

\newcommand{\R}{\mathbb{R}}

\newcommand{\floor}[1]{\lfloor #1 \rfloor}

\newcommand{\rcd}{\mathrm{RCD}}

\newcommand{\Z}{{\rm Z}}
\newcommand{\Y}{{\rm Y}}

\renewcommand{\limsup}{\varlimsup}

\renewcommand{\tan}{{\rm Tan}}

\newcommand{\hau}{\mathcal{H}}

\newcommand{\dist}{{\sf dist}}

\theoremstyle{plain}
\newtheorem{theorem}{Theorem}[section]

\newtheorem{prop}[theorem]{Proposition}

\newtheorem{cor}[theorem]{Corollary}

\theoremstyle{definition}
\newtheorem{definition}[theorem]{Definition}
\newtheorem{remark}[theorem]{Remark}

\date{}

\address{Mathematics Institute \\
	Zeeman Building \\
	University of Warwick \\
	Coventry, CV4 7AL.}
\email{matthew.hyde@warwick.ac.uk}

\address{Mathematics Research Unit, University of Oulu (Oulu, Finland)
and Departament de Mat\'{e}matiques, Universitat Aut\`onoma de Barcelona (Barcelona, Catalunya).}

\email{michele.villa@oulu.fi, michele.villa@uab.cat}

\address{Centro di Ricerca Matematica Ennio De Giorgi, Scuola Normale Superiore  (Pisa , Italy).}

\email{ivan.violo@sns.it}

\thanks{
\noindent
 M.H. is supported by the European Union’s Horizon 2020 research and
innovation programme (Grant agreement No. 948021). M.V. is supported by the European Research Council (ERC) under the European Union's Horizon 2020 research and innovation programme (grant agreement 101018680, with PI X. Tolsa), and by the Academy of Finland via the project \emph{Higher dimensional Analyst's Traveling Salesman theorems and Dorronsoro estimates on non-smooth sets}, grant No. 347828. I.Y.V. was partially supported by the Academy of Finland project \textit{ Incidences on Fractals}, Grant No.321896.
\vspace{0.2cm}
}

\numberwithin{equation}{section}

\begin{document}

\begin{abstract}
	We prove that  Ahlfors-regular RCD spaces are uniformly rectifiable. The same is shown for  Ahlfors regular boundaries of  non-collapsed RCD spaces. As an application we deduce a type of quantitative differentiation  for Lipschitz functions on these spaces.
\end{abstract}

\maketitle


\section{Introduction}
The aim of this note is to provide some concrete examples of \textit{uniformly rectifiable} metric spaces. More precisely, we show that a vast class of RCD spaces, are uniformly rectifiable (UR).

\subsection{Uniform rectifiability} Uniform rectifiability is a quantitative strengthening of the qualitative property of being rectifiable. Recall that if $E \subset \R^n$ has finite $\mathcal{H}^k$ measure then one says that $E$ is $k$-rectifiable if there are Lipschitz maps $f_i: \R^k \to \R^n$, $i=1,2,...$ so that 
\begin{equation*}
    \mathcal{H}^k\left(E \setminus \bigcup_{i=1}^\infty f_i(\mathbb{R}^k)\right) = 0. 
\end{equation*}
A $k$-rectifiable set $E \subset \R^n$ looks like $\R^k$ asymptotically, but we cannot say anything at any definite scale. On the other hand, $E$ being uniformly $k$-rectifiable tells us that the scales at which $E$ is non-flat, that is, very far from looking like $\R^k$, are just a few. A different but equivalent way to put this is to say: if we look at a $k$-rectifiable set $E$ in a ball $B_r(x)$, $x\in E$, then we are guaranteed that $\mathcal{H}^k(B_r(x) \cap E \cap f(\R^k))>0$, where $f:\R^k \to \R^n$ is Lipschitz - but no more. If $E$ is $k$-UR, however, we know that $\mathcal{H}^k(B_r(x) \cap E \cap f(\R^k))\geq c r^k$, and $c$ is uniform in $x$ and $r$. These `quantifications' (that can be traced back to the landmark \cite{jones90} and \cite{david-semmes91, david-semmes93}) have had far reaching consequences, at least as far as sets and measures in Euclidean space are concerned: the solution of the Painlev\`{e} problem \cite{mmv, tolsa03}; the partial one of Vitushkin's conjecture \cite{david1998unrectictifiable, chang2020analytic, dkabrowski2022analytic}; the (also partial) solution of the David-Semmes problem on the boundedness of the Riesz transform \cite{ntv}; that of Bishop conjecture on harmonic measure \cite{ahmmmtv}; that of the Dirichlet \cite{ahmmt} and regularity \cite{mourgoglou2021regularity} problems in elliptic PDEs - all these rest upon the quantifications mentioned above. We refer the reader to P.\ Mattila's survey \cite[Section 6]{mattila2023rectifiability}. Very recently, the first author together with D.\ Bate and R.\ Schul \cite{BHS23}  generalised some aspects of this theory to metric spaces\footnote{We recall, however, that there is a vast and expanding literature about quantitative and qualitative rectifiability both for specific model spaces and for the general setting. Heisenberg groups and parabolic spaces are two of the most studied models. We refer to the introduction of \cite{fassler2023various} for a rather thorough review of the literature. It should be remarked that rectifiability in those contexts is understood in terms of intrinsic objects (e.g. \textit{intrinsic} Lipschitz graph) which are unrectifiable from the Euclidean point of view. Here, on the other hand, we are interested in to what extent a metric space looks \textit{Euclidean}. For one-dimensional metric spaces, the theory is rather well developed, see \cite{fassler2023various} and \cite{mattila2023rectifiability}.}. This note confirms that there are plenty of (relevant) UR metric spaces. 

\subsection{Spaces with lower bounds on the Ricci curvature}
The metric spaces which are the focus of this note are \textit{RCD spaces}. Roughly speaking, RCD spaces are a class of metric measure spaces (m.m.s) which have the defining property of having lower bounds on (a synthetic notion of) Ricci curvature. This condition originated in the study of Riemannian manifolds and from the fundamental question of how lower bounds on curvature (be it Ricci, sectional, scalar) affect their global and local geometry. Restricting now the discussion to manifolds with Ricci lower bounds, in order to study their local properties it comes natural to take sequences of such manifolds, in some appropriate sense, and study whatever limiting object (a \textit{Ricci limit}) is found. In this sense, it was observed by Gromov \cite{Gr99} that the family of $n$-dimensional Riemannian manifolds having Ricci curvature bounded below by $K\in \rr$ and diameter bounded above by $D<\infty$, is pre-compact in the Gromov-Hausdorff topology. The study of \textit{Ricci limits} went through a major development in a series of works by Cheeger and Colding in the nineties  \cite{CC1,CC2,CC3,CC96} (see also the survey \cite{wei2006manifolds}). By then, however, it was still unclear whether a notion of \textit{intrinsic} Ricci curvature lower bound could be defined, so as to impose it on a general m.m.s.\ \textit{a priori} - without having to rely on the lower bound of limiting sequences of Riemannian manifolds (see \cite[Appendix 2]{CC2}). To this end, Lott-Villani and Sturm \cite{Sturm06-2,Sturm2,LV09} independently introduced the so called \textit{Curvature Dimension} (CD) condition, which is, to all effect, a satisfactory synthetic notion of Ricci curvature lower bound in the setting of metric measure spaces. The CD condition is compatible with the smooth case, i.e.\ it coincides with the classical definition in the case of Riemannian manifolds and it is strong enough to obtain \textit{interesting} theorems. It is also sufficiently weak to be \textit{stable} under (measure) Gromov Hausdorff convergence and, in particular, includes Ricci limits. On the other hand, it is satisfied by spaces which are quite far from being Riemannian, for example Finsler manifolds such as $(\rr^n,\|\cdot \|,\hau^n)$ for any norm $\|\cdot \|$. In the last decade, a stronger condition has become rather prominent and much studied - the so called \textit{Riemannian Curvature Dimension} (RCD) condition (see the surveys \cite{Aicm,  G23} for more details and historical notes). Notable examples of RCD spaces are weighted Riemannian manifolds, Ricci limits, Alexandrov spaces with (sectional) curvature bounded below \cite{Pet11,ZZ10} and stratified spaces \cite{BKM21}. Non-Riemannian Finsler spaces, on the other hand, are ruled out by the RCD condition. Nonetheless, we stress that RCD spaces (and even Ricci limits) are singular spaces, for instance they are not necessarily locally Euclidean, can have both non-unique and non Euclidean tangent spaces (see e.g.\ \cite{CoNa13}) and admit conical singularities \cite{Ket15}. The standard notation for the class of these spaces is $\RCD(K, N)$, where $K \in \rr$ represents the (synthetic notion of) Ricci lower bound, while $N \in [1,\infty]$ is the upper bound for the dimension.

Besides the conceptual importance to develop a theory of Ricci curvature bounded below in the non-smooth setting,  RCD spaces have proven useful and even necessary in answering questions about smooth Riemannian manifolds. For example, they have been used to show existence of isoperimetric sets and  sharp concavity estimates of the isoperimetric profile in non-compact Riemannian manifolds \cite{APPS22,ABFP22}, to prove stability of sharp Sobolev inequalities under non-negative Ricci curvature and Euclidean volume growth \cite{NV22}, to show lack of uniform $C^1$-estimates for harmonic functions assuming only Ricci or sectional lower bounds \cite{DPZ23}. {Furthermore, several new almost-rigidity results for functional and geometric inequalities in Riemannian manifolds have been obtained by means of the RCD theory (see, e.g., \cite{CavM17,MonSe20,honda2023sharp,KETTERER2023113202}).}
We refer to \cite[Section 7]{G23} for more details and examples.  Notably RCD spaces found application also in the gravitational fields theory from physics, where they have been exploited to obtain  eigenvalues bounds in some singular weighted manifolds \cite{Phy1,Phy2}.

The rectifiability of RCD spaces was proved in \cite{MN19}. This was further developed in the independent works \cite{GP21,KM18,DPMR17}, where it was shown that RCD spaces are rectifiable as metric measure spaces - that is to say, the reference measure is absolutely continuous with respect to the appropriate Hausdorff measure. In fact, it follows from \cite{MN19} that RCD spaces are \textit{strongly rectifiable}, in the sense that they can be covered up to a measure zero set by $(1+\eps)$-biLipschitz images of the Euclidean space, where $\eps>0$ can be taken arbitrarily small.

In our main result  we show that in the case of bounded Ahlfors regular RCD spaces, rectifiability can be upgraded to uniform rectifiability (see Definition \ref{def:UR}). Roughly speaking, this says (using Remark \ref{rem-li-bilip} below) that any ball of an Ahlfors regular RCD space has a large portion which is bi-Lipschitz equivalent to a subset of the Euclidean space.
\begin{theorem}\label{thm:rcd ur bounded}
    Every bounded Ahlfors $k$-regular $\rcd(K,N)$ space, $N<\infty,$ is uniformly $k$-rectifiable. In particular any bounded non-collapsed $\rcd(K,N)$ space is uniformly $N$-rectifiable.
\end{theorem}
\textit{Non-collapsed} $\rcd(K,N)$ spaces are the ones  having the Hausdorff measure $\hau^N$ as reference measure (see Section \ref{sec:pre}). Note that Ahlfors regularity is part of the definition of UR, hence it is a not a restrictive assumption. We recall also that there are non-Ahlfors regular $\rcd(K,N)$ spaces, e.g.\  the  $\rcd(0,N)$ space $([0,1], t^{N-1}\d t, \sfd_{Eucl})$.

Concerning the unbounded case we can prove  uniform rectifiability under non-negative Ricci curvature lower bound, while in the general case we can still obtain a local version of Theorem \ref{thm:rcd ur bounded}, stated below. Unbounded RCD spaces are of considerable interest  and are often studied as they naturally appears  for example when taking blow-ups of RCD spaces or  blow-downs of Riemannian manifolds, therefore in this setting it is relevant to consider uniform rectifiability also in the unbounded case.  
\begin{theorem}\label{thm:rcd ur local}
A locally Ahlfors $k$-regular $\rcd(K,N)$ space is locally uniformly $k$-rectifiable. Additionally any  Ahlfors $k$-regular $\rcd(0,N)$ space is uniformly $k$-rectifiable. 
\end{theorem}
We do not expect in general (global) uniform rectifiability to hold, indeed typically for $K$ negative the constants in most functional and geometric  inequalities   degenerate at large scales, e.g.\ in the Poincaré inequality and Bishop-Gromov volume comparison \cite{Sturm06-2,Rajala12-2}. 

As a second result we obtain  the uniform rectifiability of the boundary $\partial \X$ of non-collapsed $\RCD(K,N)$ spaces  which, roughly speaking, is  the closure of the points having the half space $\rr^N_+$ as a tangent (see Section \ref{sec:pre} for details). 
\begin{theorem}\label{thm:UR boundary}
Let $(\X,\sfd,\hau^N)$ be an $\rcd(-(N-1),N)$ space such that the boundary $\partial \X$ is bounded (resp.\ unbounded) and locally $(N-1)$-Ahlfors regular. Then $\partial \X$ is   uniformly $(N-1)$-rectifiable  (resp.\ locally uniformly $(N-1)$-rectifiable).
\end{theorem}
The previous result heavily relies on the regularity results for the boundary obtained recently in  \cite{BNS22} where, among other things, it is shown that $\partial \X$ is $(N-1)$-rectifiable. 
It was also conjectured in \cite{BNS22} and proved there in the case of Ricci limits and Alexandrov spaces with curvature bounded below, that $\partial \X$ is in fact always locally $(N-1)$-Ahlfors regular, a property  which  could then be removed  as an assumption  from  Theorem \ref{thm:UR boundary} (see also Remark \ref{rmk:boundary conj}).

In Section \ref{s:quant-diff}, we deduce a statement about quantitative differentiation of Lipschitz functions on bounded RCD spaces i.e.\ we obtain information about how well Lipschitz functions are approximated at coarse scales, not just infinitesimally. We show, for any given Lipschitz $f \colon X \to \R,$ that \textit{most} balls in $X$ are GH-close to a Euclidean ball of the same radius and that $f$ is well-approximated by an affine function on the corresponding tangent space (see Corollary \ref{c:quant-diff}). This is similar in spirit to a result of Jones \cite{jones1988lipschitz} for Lipschitz functions defined on $\R^n.$ A different notion of quantitative differentiability on spaces admitting a Poincar\'e inequality (in particular, RCD spaces) has also been investigated in \cite{cheeger2012quantitative}.

\section{Preliminaries}\label{sec:pre}
A \textit{metric measure space} (m.m.s.) is a triple $\Xdm$, where $(\X,\sfd)$ is a complete and separable metric space and $\mm$ is a positive Borel measure finite on bounded set,  called \textit{reference measure}. We will always assume that $\supp(\mm)=\X.$  We will denote by $\hau^{\alpha}$ and $\hau^{\alpha}_\infty$ the $\alpha$-dimensional Hausdorff measure and Hausdorff content in $(\X,\sfd)$. For all $x\in \X$ and $r>0$ we also set $B_r^{\X}(x)\coloneqq \{y \in \X \ : \ \sfd(x,y)<r\}$, omitting the superscript $\X$ when there is no confusion in doing so. {
We will say that a subset $E$ of a metric space $(\X,\sfd)$ is \textit{($L$-)biLipschitz equivalent} to a subset $\rr^n$ if there exists an ($L$)-biLipschitz map $f:E\to \rr^n$. }

Given  $\alpha>0$ we say that a closed subset $E\subset \X$ of a metric measure space $\Xdm$ is \textit{locally Ahlfors $\alpha$-regular} if for every bounded set $B\subset \X$ there exists a constant $C_B\ge 1$ and a radius $R_B>0$ such that 
 \begin{equation}\label{eq:A-regular}
     C_B^{-1} r^\alpha \leq \mm(B_r(x)\cap E) \leq  C_B r^\alpha, \quad \forall\,  x \in E\cap B,\quad \forall\,  0 < r < R_B.
 \end{equation}
  If $C_B\equiv C$ can be taken independent of $B$ and $R_B=\diam(E)$ for all $B$ we say that $E$ is Ahlfors $\alpha$-regular and we call the minimal such  $C$ the Ahlfors regularity constant of $E.$ Our definition of local Ahlfors regularity, where the constant might depend on the location of the space, is motivated by the fact that any non-collapsed $\rcd(K,N)$ space is automatically locally Ahlfors $N$-regular (see \eqref{eq:non coll reg} below).
 By standard facts about differentiation of measures (see \cite[Theorem 2.4.3]{ATbook}), if  $E$ is locally Ahlfors $\alpha$-regular  then for all bounded sets $B\subset \X$
 \begin{equation}\label{eq:euiv hau}
     (c_{\alpha,B})^{-1} \hau^\alpha\restr {E\cap B} \le \mm\restr {E\cap B}\le c_{\alpha,B} \hau^\alpha\restr {E\cap B},
 \end{equation}
 where $c_{\alpha,B}\ge 1$ is a constant depending only on $C_B$ and $\alpha.$ In particular $E$ is (locally) Ahlfors regular in the space $\Xdm$ if and only if it is so in $(\X,\sfd,\hau^\alpha).$

\begin{definition}[Uniform rectifiability]\label{def:UR}
    	A closed subset $E\subset \X$ of a metric measure space $\Xdm$\footnote{Recalling \eqref{eq:euiv hau}, in Definition  \ref{def:UR}  is equivalent to consider $\hau^k$ in \eqref{eq:UR}  in place of $\mm.$ In particular our notion of UR is indeed equivalent to the one given in \cite[Def.\ 1.2.1]{BHS23}.} 
     is said to be \textit{locally uniformly $k$-rectifiable (locally UR)} if it is locally Ahlfors $k$-regular  and has \textit{locally Big Pieces of Lipschitz Images (locally BPLI) of $\R^k$} i.e.\ for every bounded set $B\subset \X$ there exist constants $\theta_B>0, L_B \geq 0$ such that for each $x \in E\cap B$ and $0 < r < R_B$ there is a set $F\subset B^{\rr^k}_r(0)$ and  an $L_B$-Lipschitz map $f \colon F \to E$ such that 
     \begin{equation}\label{eq:UR}
         \mm(B_r(x) \cap f(F)) \geq \theta_B r^k.
     \end{equation}
      Moreover if $E$ is Ahlfors $k$-regular and has Big Pieces of Lipschitz Images (BPLI) of $\rr^k$, that is to say it has locally BPLI of $\rr^k$ but we can take the constants $\theta_B, L_B$  to be independent of $B$ and $R_B=\diam(E)$ for all $B$, we say that $E$  is uniformly $k$-rectifiable (UR).
\end{definition}
Note that if $E$ is locally uniformly $k$-rectifiable, then for any $B$ we can take $R_B$ arbitrarily large, up to decreasing the constant $\theta_B$. In particular any bounded Ahlfors $k$-regular set which is locally uniformly $k$-rectifiable is in fact uniformly $k$-rectifiable.

\begin{remark}\label{rem-li-bilip}
    It is proved in \cite{Scul09} that in Definition \ref{def:UR}   it  is equivalent to require that $\Xdm$ has Big Pieces of \textit{biLipschitz} Images (BPBI) of $\R^k$, i.e.\ that the function $f$ is $L_B$-biLipschitz.
\end{remark}

We assume the reader to be familiar with the notion of \textit{Gromov-Hausdorff distance}, $\sfd_{GH},$ together with \textit{pointed measure Gromov-Hausdorff} (pmGH) \textit{convergence and distance}, $\sfd_{pmGH},$ referring to \cite{GMS15} for the relevant background.  We will call a map $g:(\X_1,\sfd_1)\to (\X_2,\sfd_2)$ between two metric spaces an \textit{$\eps$-isometry }for some $\eps>0$ if 
\[
|\sfd_2(g(x),g(y))-\sfd_1(x,y)|\le \eps, \quad \forall x,y \in \X_1.
\]
If $g$ is both an $\eps$-isometry and  $g(\X_1)$ is  $\eps$-dense in $\X_2$, we say that $g$ is a \textit{$\eps$-GH isometry}.

We say that the pointed m.m.s.\  $(Y,\sfd_Y,\mu,y)$, $y\in \Y,$ is a \textit{tangent} to $(\X,\sfd,\mm)$ at  $x\in \X$ if there exists a sequence $r_n\to 0^+$ such that $(\X,r_n^{-1}\sfd,(c_{r_n}^{ x})^{-1}\mm, x)$ pmGH-converges  to  $(Y,\sfd_Y,\mu,y)$, where 
\begin{equation}\label{eq:renorm}
    c_{r}^{ x}\coloneqq \int_{B_r(x)} 1-\frac{\sfd(\cdot ,  x)}{r} \d \mm.
\end{equation}
Tangents are not necessarily unique, and we denote by $\tan(\X,\sfd,\mm,x)$ the class of all tangents to $\Xdm$ at $x\in \X$. The \textit{$k$-dimensional regular set} $\mathcal R_k$ is given by 
\[
\mathcal R_k\coloneqq \left  \{x \in \X \ : \ \tan(\X,\sfd,\mea,x)=\{(\rr^k,\sfd_{Eucl},c_k\mathcal H^k,0^k)\}  \right\},
\]
where $c_k\coloneqq \int_{B_1(0^k)} (1-|y|)\d \mathcal H^k(y).$ A pointed metric measure space $(\X,\sfd,\mea,x)$ satisfying  $c_1^x=1$ is called \textit{normalised}.

For an introduction to the theory of $\rcd(K,N)$ spaces we refer to the surveys \cite{Aicm,G23} and references therein, limiting ourselves here to recall the main results that we need. 
\begin{remark}\label{rmk:compactness}
    A key property of $\rcd$ spaces is compactness  in the pmGH-topology, i.e.\ a sequence $(\X_n,\sfd_n,\mm_n,x_n)$ of pointed $\rcd(K,N)$ spaces satisfying $\mm_n(B_1(x_n))\in [v^{-1},v]$ for some $v\ge1$, admits a subsequence converging in the pmGH-sense to a limit  $\rcd(K,N)$ space
    $(X_\infty,\sfd_\infty,\mea_\infty,x_\infty)$.
\end{remark}
A basic scaling property is:
\begin{equation}\label{eq:rcd scaling}
    \Xdm \text{ is an $\rcd(K,N)$ space} \implies (\X,r^{-1}\sfd,\mea) \text{ is an $\rcd(r^2K,N)$ space}.
\end{equation}
The following is part of the now well established  structural and rectifiability properties of $\rcd$ spaces.
\begin{theorem}[\hspace{1sp}\cite{MN19,AHT18}]\label{thm:structural}
      Let $\Xdm$ be an  $\rcd(K,N)$ space with $N<\infty$. Then 
    \[
    \mm\left( \X\setminus \bigcup_{k=1}^{\floor N} \mathcal R_k\right)=0
    \]
    and $\lim_{r\to 0^+} \frac{\mea(B_r(x))}{r^k}\in (0,\infty)$  for $\mea$-a.e.\ $x \in \mathcal R_k$, for all $k=1,\dots,\floor N.$
\end{theorem}
Actually, as shown in \cite{BS20}, $\mm(\mathcal R_k)=0$ for all except one $k$, however Theorem \ref{thm:structural} will suffice for our purposes. 

 According to the notation introduced in \cite{DPG18}, an $\rcd(K,N)$ space endowed with the reference measure  $\mm=\hau^N$ is said to be \textit{non-collapsed}. It is also shown in \cite{DPG18} that $N$ must be an integer. 

\begin{remark}
    It is worth to mention that if $(\X,\sfd,\hau^n)$ is an $\rcd(K,N)$ space which is also $n$-locally Ahlfors regular (as usually assumed here) then it is in fact  an $\rcd(K,n)$ space and in particular it is non-collapsed in the terminology above (see \cite{H20,BGHZ23}). It is conjectured  in \cite{H20bis} that this actually holds without assuming local Ahlfors regularity.
\end{remark}
 
 For a non-collapsed $\RCD(K,N)$ space $(\X,\sfd,\hau^N)$ we also recall the notion of \textit{$k$-singular set}, for all $0\le k\le N-1$,
 \[
 \mathcal S^k\coloneqq \{x \in\ X \ : \ \text{no element of $\tan(\X,\sfd,\mea,x)$ splits off $\rr^{k+1}$ isometrically}\}.
 \]
The $k$-singular sets are nested and induce the following stratification
\[
\mathcal S^0\subset \mathcal S^1\subset \dots \subset \mathcal S^{N-1}=\X\setminus \mathcal R_N.
\]
 It was proved in \cite{DPG18} that 
\begin{equation}\label{eq:sk dim bound}
    \dim_H(\mathcal S^k)\le k, \quad \text{ for all } 0\le k\le N-1.
\end{equation}
The \textit{boundary} of a non-collapsed $\RCD(K,N)$ space $\Xdm$ was introduced in \cite{DPG18} as
\[
\partial \X \coloneqq \overline{\mathcal S^{N-1}\setminus \mathcal S^{N-2}}.
\]
It easily follows from the definition that $\partial \X \setminus ({\mathcal S^{N-1}\setminus \mathcal S^{N-2}})\subset \mathcal S^{N-2}$ (see e.g.\ \cite[(1.10)]{BNS22}), hence it also follows that
\begin{equation}\label{eq:N-2}
    \dim_H\big(\partial \X \setminus ({\mathcal S^{N-1}\setminus \mathcal S^{N-2}})\big)\le N-2.
\end{equation}
As observed in \cite[Lemma 4.6]{MK21} it  holds
\begin{equation}\label{eq:half space tangent}
(\rr^N_+,\sfd_{Eucl}, c_N/2\mathcal H^N,0^N)\in \tan(\X,\sfd,\mea,x), \quad \text{ for all $x \in \mathcal S^{N-1}\setminus \mathcal S^{N-2}$}.
\end{equation}
In particular by the volume convergence theorem \cite[Theorem 1.3]{DPG18} we have both 
\begin{equation}\label{eq:theta}
     \begin{split}
     &\theta_N(x)=1, \quad \forall x \in \mathcal R_N,\\
     &\theta_N(x)=\frac12, \quad \forall x \in  \mathcal S^{N-1}\setminus \mathcal S^{N-2}.
 \end{split}
\end{equation}
where $\theta_N(x)\coloneqq \lim_{r\to 0^+} \frac{\hau^N(B_r(x))}{\omega_N r^N}$, which exists by the Bishop-Gromov inequality. By standard measure theory (see \cite[Theorem 2.4.3]{ATbook}) it holds that $\theta_N(x)\le 1$ for all $x\in \X$. This combined with the Bishop-Gromov inequality \cite{Sturm06-2} shows that 
\begin{equation}\label{eq:non coll reg}
    \parbox{13cm}{a non-collapsed $\rcd(K,N)$ space $(\X,\sfd,\hau^N)$ is locally Ahlfors $N$-regular.}
\end{equation}

The key tool that we will use in the sequel is the one of harmonic $\delta$-splitting maps, which play the role of coordinate-functions. These type  of maps were introduced in \cite{CC96} in the context of Ricci limit spaces (see also \cite{CC1,CC2,CC3}) and have been extended to the $\RCD$ setting in \cite{BPS23bis} (see also \cite{BPS21,BNS22} for further developments).
\begin{definition}
	Let \((X,\sfd,\mm)\) be an \(\RCD(-(N-1),N)\) space. Let \(x\in X\) and \(\delta>0\)
	be given. Then a map \(u=(u_1,\ldots,u_k)\colon B_r(x)\to\R^k\) is said to be
	a \emph{\(\delta\)-splitting map} provided:
	\begin{itemize}
		\item[\(\rm i)\)] \(u_a\colon B_r(x)\to\R\) is harmonic and \(C_N\)-Lipschitz
		for every \(a=1,\ldots,k\),
		\item[\(\rm ii)\)] \(r^2\fint_{B_r(x)}\big|{\rm Hess}(u_a)\big|^2\,\d\mm\leq\delta\)
		for every \(a=1,\ldots,k\),
		\item[\(\rm iii)\)] \(\fint_{B_r(x)}|\nabla_\mm u_a\cdot\nabla_\mm u_b
		-\delta_{ab}|\,\d\mm\leq\delta\) for every \(a,b=1,\ldots,k\).
	\end{itemize}
\end{definition}
We will never explicitly use  the above definition of splitting maps, but we will instead exploit some of their properties listed in the result below. For this reason we will avoid introducing the notion of gradient and Hessian in the metric setting and refer to \cite{GP20} for details.
\begin{theorem}[Properties of $\delta$-splitting maps]\label{thm:delta split and eps isom}
    For every $N\in [1,\infty),$ $C\ge 1$ and $\eps \in(0,1/2)$ there exists $\delta=\delta(N,\eps,C)\in(0,1)$ such that the following hold. Let $(\X,\sfd,\mm)$ be an   $\rcd(-\delta,N)$ space. Then for all  $x \in \X $ and $r\in(0,1]$ it holds:
    \begin{enumerate}[label=\roman*)]
    \item\label{it:propagation}   if \(u\colon B_{2r}(x)\to\R^k\) is an
	\(\eta\)-splitting map, with $\eta\in(0,1)$, then there exists a Borel set \(G\subseteq B_r(x)\) such that
	\(\mm\big(B_r(x)\setminus G\big)\leq C_N\sqrt\eta\,\mm\big(B_r(x)\big)\) and
	$
	u\colon B_s(y)\to\R^k$ is a $\sqrt{\eta}s$-splitting map for every $y\in G$ and $s\in(0,r),$
        \item\label{it:eps to delta} if 
        \begin{equation}
            \sfd_{pmGH}((\X,r^{-1}\sfd,\mea,x)),(\rr^k,|\cdot|,c\mathcal H^k,0^k))\le \delta,
        \end{equation}
         for some constant $c>0$, then there exists an $ \eps r$-splitting map $u:B_{r}(x)\to \rr^k$,
         \item\label{it:delta to eps} assuming furthermore that
         \begin{equation}\label{eq:AD-regular in delta split prop}
        C^{-1}s^k\le \mm(B_s(y))\le Cs^k, \quad \text{for all $y\in B_r(x)$ and $s\in (0,r]$.}
    \end{equation}
    and that $u:B_{6r}(x)\to \rr^k$  is $\delta r$-splitting map,
         then
         \begin{equation}
             \sfd_{GH}(B_\frac{r}{k}(x),B^{\rr^k}_\frac{r}{k}(0))\le \eps r
         \end{equation}
         and $u: B_\frac{r}{k}(x) \to \rr^k$ is an $\eps r$-isometry.
    \end{enumerate}
\end{theorem}

\begin{proof}
Item $\ref{it:propagation}$ is just Proposition 1.6 in \cite{BPS21}.
    For items $\ref{it:eps to delta}$ and $\ref{it:delta to eps}$ by the scaling property of  $\delta$-splitting maps  we can assume that $r=1$. Item $\ref{it:eps to delta}$ is then precisely \cite[Proposition 3.9]{BPS23bis}.   Item $\ref{it:delta to eps}$ follows instead from $ii)$ in \cite[Theorem 3.8]{BNS22}, the only difference is that therein   the space is assumed to be normalised and the conclusion is that  $(u,f):B_{1/k}(x)\to  \rr^k\times Z$ is an $\eps $-isometry for some $(Z,\sfd_Z)$ and some $f:B_{1/k}(x)\to Z$. We explain now how the argument in \cite[Theorem 3.8]{BNS22} gives also the version stated here (cf.\ also with \cite[Remark 3.10]{BNS22}). As in \cite{BNS22} by contradiction we assume the existence of  a sequence $(\X_n,\sfd_n,\mea_n,x_n)$ of $\rcd(-1/n,N)$ spaces such that \eqref{eq:AD-regular in delta split prop} holds (with $x=x_n$, with $k=k_n\in \nn$ and  the same constant $C>0$)    and also of $1/n$-splitting maps $u_n: B_{6}(x_n)\to \rr^{k_n}$ such that  either $u_n:B_{1/k}(x_n) \to \rr^k$ is not an $\eps $-isometry or $\sfd_{GH}(B_{1/k}(x),B^{\rr^k}_{1/k}(0))> \eps $. Since $k_n\le N$ (e.g.\ by  Theorem \ref{thm:structural}), up to a subsequence we can assume that $k_n\equiv k\in \nn$.  Since $\mm_n(B_1(x_n)) \in [C^{-1},C]$ (which replaces the normalised assumption), up to a further subsequence, $(\X_n\,\sfd_n,\mea_n,x_n)$ pmGH-converge to some $\rcd(0,N)$ space $(X_\infty,\sfd_\infty,\mea_\infty,x_\infty)$ (recall Remark \ref{rmk:compactness}). In particular $\X_\infty$ still satisfies \eqref{eq:AD-regular in delta split prop} with $x=x_\infty$ and $r=1$, which  implies $\dim_H(B_{1}(x_\infty))\le k$  (see \cite[Theorem 2.4.3]{ATbook}). Proceeding now verbatim as in \cite{BNS22} we can find a map $( u, f):B_{1/k}(x_\infty)\to \rr^k\times Z$ which is an isometry with its image (for some $(Z,\sfd_Z)$ and some $f$) and deduce for $n$ big enough  that $\sfd_{GH}(B_{1/k}(x),B^{\rr^k\times Z}_{1/k}(0))\le  \eps $ and  $(u_n,f_n): B_{1/k}(x_n)\to \rr^k\times Z$ is an $\eps$-isometry for some $f_n$  and some $z\in \Z$ independent of $n$ (see respectively eq.\ $(3.34)$ and $(3.35)$ in \cite{BNS22}).  It is therefore enough to show that $ f(B_{1/k}(x_\infty))= \{ f(x_\infty)\},$ indeed we could then  replace $Z$ with $\{ f(x_\infty)\}$ and get the desired contradiction. To show this we note that  $( u, f)$ is obtained in \cite{BNS22} applying Theorem 3.4 therein and, inspecting its proof, we see that $( u, f)\coloneqq \Phi^{-1}\restr{B_{1/k}(x_\infty) }$ where $\Phi: (-1/k,1/k)^k\times B_1(q)\to \X$ (for some $q\in \Z$) is an isometry  with its image, which contains $B_{1/k}(x_\infty)$. Therefore $( u, f)(B_{1/k}(x_\infty))$ is open in $\rr^k\times Z$. If by contradiction $ f(x_\infty)\neq  f(x)$ for some $x \in B_{1/k}(x_\infty)$ the $Z$-component of  $( u, f)(\gamma)$, where $\gamma$ is a geodesic from $x_\infty$ to $x$ (recall that $\X_\infty$ is geodesic), is itself a geodesic from $ f(x_\infty)$ to $ f(x)$ (see e.g.\ Lemma 3.6.4 in \cite{BBI01}). Hence  $ \hau^1(B_r^Z(f(x_\infty)))>0$ for all $r>0$. Thus, since $ ( u, f)(B_{1/k}(x_\infty))$ is  open, it contains $B^{\rr^k}_s(u(x_\infty))\times B^Z_s(f(x_\infty))$ for some $s>0$ and thus it has positive $\mathcal H^{k+1}$-measure (see \cite[Theorem 2.10.45]{Federer}). However as observed above $\dim_H(B_{1/k}(x_\infty))\le k$ which contradicts the fact that  $( u, f)$ is an isometry.
\end{proof}

\section{Proof of the results}
\subsection{Uniform rectifiability of Ahlfors regular RCD spaces}
The proof is a  combination of two results. The first (Proposition \ref{prop:jonas lemma}) says that in an Ahlfors $k$-regular $\rcd$ space every ball contains another ball of comparable size that is \textit{almost-flat}, i.e.\ Gromov-Hausdorff close to an Euclidean ball in $\rr^k$ of the same radius. The second (Proposition \ref{prop:flat implies bilip})  shows that a big portion of an almost-flat ball is covered by a biLipschitz image of $\rr^k.$

The  result below is inspired by \cite[Lemma 2.4]{A21}, where a similar statement is shown for Ahlfors regular sets in the Euclidean space supporting a Poincaré inequality. 
\begin{prop}[Existence of many large almost flat balls]\label{prop:jonas lemma}
    For every $\eps>0$, $C\ge 1$ and $N\in[1,\infty)$  there exists $\eta=\eta(\eps,C,N)>0$ such that the following holds. Let $\Xdm$ be an $\rcd(-(N-1),N)$ space and $x\in \X$ be such that for some $k\in\nn$ it holds
     \begin{equation}\label{eq:AD-regular in jonas lemma}
        C^{-1}s^k\le \mm(B_s(y))\le Cs^k, \quad \text{for all $y\in B_1(x)$ and $s\in (0,1]$.}
    \end{equation}
    Then there exists $x'\in B_{\frac12}(x)$, $r_0\in(\eta,1/2)  $ and $c>0$  such that
    \begin{equation}\label{eq:large almost flat ball}
        \sfd_{pmGH}((\X,{r_0}^{-1}\sfd,\mea,x'),(\rr^k,|\cdot|,c\mathcal H^k,0^k))\le \eps .
    \end{equation}
\end{prop}
\begin{proof}
    Suppose by contradiction that there exist $\eps>0$, $C>0$, $N\in [1,\infty)$, a sequence of  pointed $\rcd(-(N-1),N)$ spaces $(\X_n,\sfd_n,\mea_n,x_n)$ and integers $k_n\in \nn$ such that \eqref{eq:AD-regular in jonas lemma} holds with $x=x_n$, $k=k_n$ and such that for all $x'\in B_{1/2}(x_n)$, all $r\in( 1/n,1)$ and all $c>0$ it holds 
 \begin{equation}\label{eq:contradiction ass}
      \sfd_{pmGH}((\X_n,{r}^{-1}\sfd_n,\mea_n,x_n'),(\rr^k,|\cdot|,c\mathcal H^k,0^k))> \eps .
 \end{equation}
As $k_n\le N$ (recall Theorem \ref{thm:structural}) up to a subsequence we can assume that $k_n\equiv k\in \nn.$
Since by assumption  it holds $\mm_n(B_{1}(x_n))\in[C^{-1},C]$, up to a subsequence, we have that $(\X_n,\sfd_n,\mea_n,x_n)$ converge in the pmGH-sense to an $\RCD(-(N-1),N)$ space $(\X_\infty,\sfd_\infty,\mea_\infty,x_\infty)$ (recall Remark \ref{rmk:compactness}).  In particular $\X_\infty$ still satisfies \eqref{eq:AD-regular in jonas lemma} with $x=x_\infty$. Therefore by Theorem  \ref{thm:structural}  we deduce that for $\mea_\infty$-a.e.\   $x\in B_1(x_\infty)$ it holds $\tan(\X,\sfd,\mea,x)=\{(\rr^k,|\cdot|,c_k\mathcal H^k,0^k)\}$. Hence we can find $x\in B_{1/4}(x_\infty)$ and  $s\in(0,1)$ such that 
\begin{equation}\label{eq:contr trigger}
    \sfd_{pmGH}((\X_\infty,s^{-1}\sfd_\infty,\mea_\infty,x),(\rr^k,|\cdot|,c_k\cdot c_s^{y}\mathcal H^k,0^k))\le \frac \eps2,
\end{equation}
where $c_s^{y}>0$ is as in \eqref{eq:renorm}.
On the other hand by the pmGH-convergence we can find points $x_n'\in \X_n$ such that $(\X_n,\sfd_n,\mea_n,x_n')$ pmGH converge to $(\X_\infty,\sfd_\infty,\mea_\infty,x_\infty)$  (see e.g.\ \cite[eq.\ (2.2)]{DPG18}). In particular $(\X_n,s^{-1}\sfd_n,\mea_n,x_n')$ pmGH converge to $(\X_\infty,s^{-1}\sfd_\infty,\mea_\infty,x_\infty)$. However recalling  \eqref{eq:contr trigger},  since $x_n' \in B_{1/2}(x_n)$ and $s>1/n$ for $n$ big enough, this gives a contradiction with \eqref{eq:contradiction ass}.
\end{proof}
Note that \eqref{eq:large almost flat ball} implies in particular that
\begin{equation*}
    \sfd_{GH}(B_{r_0}(x'),B^{\rr^{k}}_{r_0}(0))\lesssim  \eps r_0,
\end{equation*}
however \eqref{eq:large almost flat ball} takes into account also the measure and is easier to work with, in particular when applying Theorem \ref{thm:delta split and eps isom}.

The result below rests on the now well known fact that   almost-splitting maps propagate well and are biLipschitz on a large sets. Similar results  have appeared frequently in the theory of rectifiability of spaces with Ricci curvature lower bounds (see \cite{MN19,BPS21,BNS22,CMT21,CC3,BPS23bis,Jansen,KW} and also Proposition \ref{prop:boundary implies bilip} below).
\begin{prop}\label{prop:flat implies bilip}
      For every $N\in[1,\infty)$, $C\ge 1$, and $\eps \in(0,1/2)$ there exist constants $\delta=\delta(N,\eps,C)\in(0,1)$ and $ \tilde c_N>0$ such that the following hold. Let $(\X,\sfd,\mm)$ be an $\rcd(-\delta,N)$ space, $k\in \nn$, $x\in \X$ a be point and  $r\in(0,1)$ be a radius satisfying 
   \begin{equation}\label{eq:ad regular + flat}
   \begin{split}
       &C^{-1}s^k\le \mm(B_s(y))\le Cs^k, \quad \text{for all $y\in B_r(x)$ and $s\in (0,r]$},\\
       &\sfd_{pmGH}((\X,r^{-1}\sfd,\mea,x)),(\rr^k,|\cdot|,c\mathcal H^k,0^k))\le \delta, \quad \text{ for some $c>0$.}
   \end{split}
   \end{equation}
    Then there exists a  set $U\subset B_{r}(x)$ with $ \mm(U)\ge \tilde c_N\mm(B_{r}(x))$ and such that
    \begin{enumerate}[label=\roman*)]
        \item\label{it:U bilip} $U$ is  $(1+\eps)$-biLipschitz equivalent to a subset of $\rr^k$,
        \item\label{it:U BWGL} \begin{equation}\label{e:U bilip}
        \sfd_{GH}(B_{s}(y),B^{\rr^{k}}_s(0))\le s\eps, \quad \text{for all $s\in \big (0,\frac r{12k}\big)$ and  $y \in U$}.
    \end{equation}
    \end{enumerate}
\end{prop}
\begin{proof}
We can assume that $r=1,$ otherwise we can consider $(\X,r^{-1}\sfd,r^{-k}\mea)$ which satisfy the same hypotheses with the same constants $C$ and $\delta$, but with $r=1.$
Fix a constant $\tau=\tau(\eps,N,C)\in(0,1)$ small enough to be chosen later.
 If $\delta$ is chosen small enough with respect to $\tau$, by  $\ref{it:eps to delta}$ in Theorem \ref{thm:delta split and eps isom} there exists a $\sqrt \tau$-splitting map $u:B_1(x)\to \rr^k$. Then  applying $\ref{it:propagation}$ of the same theorem we obtain a set $G\subset B_{\frac {1}2}(x)$  with 
    \begin{equation}\label{eq:G big}
        \mm(G)=\mm(B_{\frac {1}2}(x))-\mm(B_{\frac {1}2}(x)\setminus G)\ge (1-C_N\sqrt{\tau})\mm(B_{\frac {1}2}(x)),
    \end{equation}
    such that for all $y \in G$ and all $s\in (0,1/2)$ the function $u:B_{s}(y)\to \rr^k$ is an $\sqrt \tau s $-splitting map. 
     Thus if $\tau$ (and thus $\delta$) are small enough we can apply $\ref{it:delta to eps}$ in Theorem \ref{thm:delta split and eps isom} and deduce that for all $s \in  (0,\frac{1}{12k})$ it holds  $\sfd_{GH}(B_{s}(y),B^{\rr^{k}}_s(0))\le s\eps$ and  $u:B_{s}(y)\to \rr^k$ is an $\eps s$-isometry. 
    Moreover combining \eqref{eq:G big} with  the Bishop-Gromov inequality \cite{Sturm06-2} and since $k\le N$, eventually leads to
\[
 \mm(G\cap B_{\frac {1}{24k}}(x))\ge \tilde c_N \mm(B_{1}(x)),
\]
    provided $\tau$ is small enough and where $\tilde c_N>0$ is a constant depending only on $N.$
    We take $U\coloneqq G\cap B_{\frac {1}{24k}}(x)$. Then item $\ref{it:U BWGL}$ holds by what we said above.
        For $\ref{it:U bilip}$ consider $y,z \in U$  arbitrary and note that $\sfd(y,z)< \frac{1}{12k}.$ Therefore  the map $u:B_{\sfd(y,z)}(y)\to  \rr^k$ is a $\sfd(y,z)\eps$-isometry, which implies
    \[
    ||u(y)-u(z)|-\sfd(y,z)|\le \eps \sfd(y,z).
    \]
    This shows that $u: U\to \rr^k$ is $(1-\eps)^{-1}$-biLipschitz and  concludes the proof of $\ref{it:U bilip}.$ 
\end{proof}

Combining the two above results we can now obtain our main technical proposition, from which the main theorems will easily follow.
\begin{prop}\label{prop:explicit}
     For every  $N\in[1,\infty)$, $C\ge 1$ and $\eps>0$ there exists a constant $\delta=\delta (N,C,\eps)\in(0,1)$   such that the following holds. Let $\Xdm$ be an $\rcd(K,N)$ space, $k\in \nn$ and  $x\in \X$,  $r\in(0,\sqrt{\delta/|K|}  )$ be a point and a radius satisfying
     \begin{equation}\label{eq:AD-regular thm}
        C^{-1}s^k\le \mm(B_s(y))\le Cs^k, \quad \text{for all $y\in B_r(x)$ and $s\in (0,r]$,}
    \end{equation}
       Then there exists a  set $U\subset B_{r}(x)$ with $ \mm( U)\ge \delta \mm(B_{r}(x))$ and such that $U$ is  $(1+\eps)$-biLipschitz equivalent to a subset of $\rr^k$. 
\end{prop}
\begin{proof}
We can assume that $r=1$ and $K=-\delta$. Otherwise we can just consider the rescaled space $(\X,r^{-1}\sfd,\mea)$ which is an $\rcd(-\delta,N)$ space, thanks to the assumption $r<\sqrt{\delta/|K|}$ (recall \eqref{eq:rcd scaling}). Then directly combining Proposition \ref{prop:jonas lemma} and Proposition \ref{prop:flat implies bilip} we can find a ball $B_{r_0}(x')\subset B_1(x)$ with $\eta(\eps,C,N)<r_0<1$ and a set $U\subset B_{r_0}(x')$ satisfying $\mm(U)\ge \tilde c_N \mea(B_{r_0}(x'))$ and which is $(1+\eps)$-biLipschitz to a subset of $\rr^k$.  Moreover by the Bishop-Gromov inequality \cite{Sturm06-2} we have $\mea(B_{r_0}(x'))\ge C_N r_0^N \mea (B_1(x))$, which concludes the proof of $i).$
\end{proof}
From Proposition \ref{prop:explicit} we immediately obtain Theorem \ref{thm:rcd ur bounded} and Theorem \ref{thm:rcd ur local}.
\begin{proof}[Proof of Theorem \ref{thm:rcd ur bounded} and Theorem \ref{thm:rcd ur local}]
It suffices to show Theorem \ref{thm:rcd ur local}. In deed the first part of   Theorem \ref{thm:rcd ur bounded} would follow recalling that if $\Xdm$ is  bounded, Ahlfors $k$-regular and locally uniformly $k$-rectifiable,   it is in fact $k$-rectifiable. From this also the second part of Theorem \ref{thm:rcd ur bounded} follows since  a bounded non-collapsed $\RCD(K,N)$ space is Ahlfors $N$-regular (recall \eqref{eq:non coll reg}).
  Let now $\Xdm$ be a locally Ahlfors $k$-regular $\rcd(K,N)$ space and fix any $\eps\in(0,1)$. Then for any bounded set $B\subset \X$ there exist $C_B\ge 1$ and $R_B>0$ such that \eqref{eq:AD-regular thm}  holds for all $x\in B$  and $r<R_B$ with $C=C_B$.  Denoted by $\delta=\delta(N,C_B,\eps)$ the constant given by Proposition \ref{prop:explicit}, we deduce that for all $x\in B$  and $r<r_B\coloneqq \min(R_B,\sqrt{\delta/|K|})$ there exists $U\subset B_r(x)$ with  $ \mm( U)\ge \delta \mm(B_{r}(x))\ge \delta C_B r^{-k}$ such that $U$ is  $(1+\eps)$-biLipschitz equivalent to a subset of $\rr^k$. This shows that $\Xdm$ has locally BPBI of $\rr^k$ and so it is locally uniformly $k$-rectifiable. This  proves the first part of  Theorem \ref{thm:rcd ur bounded}. For the second part note that if  $\Xdm$ is Ahlfors $k$-regular we can take $C_B$ independent of $B$ and $R_B=\diam (\X)$ for all $B$. Thus, if also $K=0$, we have $r_B=\diam(\X)$ and so $\Xdm$ has  BPBI of $\rr^k$ and in particular is uniformly $k$-rectifiable
\end{proof}

\begin{remark}
    As it is evident from the proof of Theorem \ref{thm:rcd ur bounded} and Theorem \ref{thm:rcd ur local}, we actually prove a slightly stronger version of uniform $k$-rectifiability, in the sense that the space has Big Pieces of $(1+\eps)$-biLipschitz Images taking $\eps>0$ arbitrarily small.
\end{remark}

\subsection{Uniform rectifiability of the boundary of RCD spaces}
The overall scheme of the argument is similar to the one showing uniform rectifiability of the ambient space, presented in the previous section. The main difference is that almost flat balls are replaced by boundary-balls, i.e.\ balls that are Gromov-Hausdorff close to the half space $\rr^N_+.$ 

\begin{prop}[Existence of many large  boundary balls]\label{prop:boundary jonas lemma}
    For every $\eps>0$, $v>0$ and $N\in \nn$  there exists $\eta=\eta(\eps,v,N)>0$ such that the following holds. Let $(\X,\sfd,\hau^N)$ be an $\rcd(-(N-1),N)$ space. Then for all $x\in \partial\X$ satisfying   $\hau^{N}(B_1(x))\ge v$, $\hau^{N-1}(\partial \X\cap B_{1/4}(x))\ge v$, 
    there exist $x'\in B_{\frac12}(x)$ and $r\in(\eta,1/2)  $  such that
    \begin{equation*}
        \sfd_{GH}(B_r(x'),B_r^{\rr^N_+}(0^N))\le \eps r.
    \end{equation*}
\end{prop}
\begin{proof}
     Suppose by contradiction that there exists $\eps>0$, $v>0$, $N\in \nn$  and a sequence of  pointed $\rcd(-(N-1),N)$ spaces $(\X_n,\sfd_n,\hau^N,x_n)$ satisfying $\hau^{N}( B_1(x_n))\ge v$ and $\hau^{N-1}(\partial \X_n\cap B_{1/4}(x_n))\ge v$,  but   such that for all $x'\in B_{1/2}(x_n)$ and all $r\in( 1/n,1)$ it holds 
 \begin{equation}\label{eq:contradiction ass boundary}
      \sfd_{GH}(B_r^{\X_n}(x'),B_r^{\rr^N_+}(0^N))> \eps r.
 \end{equation}
 By stability of non-collapsed RCD spaces \cite[Theorem 1.2]{DPG18}, up to a subsequence, $(\X_n,\sfd_n,\hau^N,x_n)$ pmGH-converges to some $\rcd(-(N-1),N)$ space $(\X_\infty,\sfd_\infty,\hau^N,x_\infty)$. We claim that $\partial \X_\infty\cap \overline B_{1/4}(x_\infty)\neq \emptyset$. To show this we follow the argument in \cite[Corollary 6.10]{BNS22}. Up to a subsequence the compact sets $C_n\coloneqq \partial \X_n\cap \overline B_{1/4}(x_n)$ converge in the Hausdorff topology to a compact set $C\subset \overline B_{1/4}(x_\infty)$  and $\hau_\infty^{N-1}(C)\ge \limsup_n \hau_\infty^{N-1}(C_n)\ge v>0. $ The lower semicontinuity of the density $\theta_N$ under pmGH-convergence \cite{DPG18}, together with \eqref{eq:theta}, then shows  that  $C\subset \mathcal S^{N-1}$ and in particular $C\setminus \partial \X_\infty \subset \mathcal S^{N-2}.$ Since $\dim_H(\mathcal S^{N-2})\le N-2$ (see \eqref{eq:sk dim bound}), we get  $\hau^{N-1}_\infty(C\setminus\partial \X_\infty)=\hau^{N-1}(C\setminus\partial \X_\infty)=0$, which gives the claim.
 Hence  we can find $y\in B_{1/3}(x_\infty)\cap ({\mathcal S^{N-1}\setminus \mathcal S^{N-2}})$ and  a radius $s\in(0,1)$ such that
 \begin{equation}
      \sfd_{GH}(B_s^{\X_\infty}(y),B_s^{\rr^N_+}(0^N))\le \eps s/2
 \end{equation}
 (recall \eqref{eq:half space tangent}).
Since there exists a sequence $x_n'\in \X_n$ such that $\sfd_{GH}(B_s^{\X_n}(x_n'),B_s^{\X_\infty}(y))\to 0$ (see e.g.\ $(2.2)$ in \cite{DPG18}) and $x_n'\in  B_{1/2}(x_n)$ for $n$ big, we obtain a contradiction with \eqref{eq:contradiction ass boundary} for $n$ big enough.
\end{proof}

\begin{remark}\label{rmk:boundary conj}
    It was conjectured in \cite{BNS22} that for a non-collapsed $\rcd(K,N)$ space $(\X,\sfd,\hau^N)$ and any $x\in \partial \X$ it holds
    \begin{equation}\label{eq:conjecture}
        \hau^{N-1}(B_2(x)\cap \partial \X)\ge C(K) \hau^{N}(B_1(x)).
    \end{equation}
    If this was true, the assumption $\hau^{N-1}(\partial \X\cap B_{1/4}(x))\ge v$ in Proposition \ref{prop:boundary jonas lemma} could be omitted. Moreover  \eqref{eq:conjecture} would also imply, by scaling, that $\partial \X$ is locally  Ahlfors $(N-1)$-regular (recall that  $\partial \X$ is  locally \textit{upper}  Ahlfors $(N-1)$-regular by \cite[Theorem 1.4]{BNS22}).
\end{remark}

The following result follows directly combining  \cite[Theorem 8.4 -$(ii)$]{BNS22} and \cite[Corollary 8.7]{BNS22}.
\begin{prop}\label{prop:boundary implies bilip}
    For every $N \in \nn $ and $\eps \in (0,1)$ there exists $\delta=\delta(\eps,N)\in(0,1)$ such that the following holds. Given any $\rcd(-\delta,N)$ space $(\X,\sfd,\hau^N)$ and a point $x\in \X$ satisfying
    \begin{equation*}
         \sfd_{GH}(B_{1}(x),B_1^{\rr^N_+}(0^N))\le \delta ,
    \end{equation*}
    there exists a set $U\subset B_{1/16}(x)\cap \partial X$ with $ \hau^{N-1}(U)\ge \frac12 \hau^{N-1}(B_{1/16}(x)\cap \partial X)$ and satisfying
    \begin{enumerate}[label=\roman*)]
        \item $U$ is $(1+\eps)$-biLipschitz equivalent to a subset of $\rr^{N-1}$,
        \item \begin{equation}\label{eq:bwgl boundary}
        \sfd_{GH}(B_{s}(y)\cap \partial \X,B^{\rr^{N-1}}_s(0))\le s\eps, \quad \forall s\in (0,1/5), \,\,\forall y \in U.
    \end{equation}
    \end{enumerate}
\end{prop}

With the above two propositions at hand we can now prove our main theorem about uniform rectifiability of the boundary.
\begin{proof}[Proof of Theorem \ref{thm:UR boundary}]
Fix $\eps\in(0,1)$ and let $\delta=\delta(\eps,N)$ be the constant given by Proposition  \ref{prop:boundary implies bilip}. 
    By the local Ahlfors $(N-1)$-regularity assumption we have that for all bounded sets $B\subset \X$ there exist $C_B\ge 1$ and $R_B>0$ such that $r^{-N}\hau^{N-1}(B_r(x)\cap \partial \X)\in [C_B^{-1},C_B]$ for all $x \in \partial \X$ with $\sfd(x,B)<1$ and all $r<R_B.$  Moreover by Bishop-Gromov monotonicity we have  $\inf_{x\in B} r^{-N}\hau^N(B_r(x))\ge v>0$ for  all $r\in(0,1)$  and some $v$ depending on $B$. Hence by (the scaling invariant version of) Proposition \ref{prop:boundary jonas lemma} we deduce that for all $x\in B\cap \partial \X$, all $r<\min(R_B,1)$  there exists $x'\in B_{r/2}(x)$ and $r'\in (\eta r, r/2)$ such that $\sfd_{GH}(B_{r'}(x'),B_r^{\rr^N_+}(0^N))\le  \delta r'$ for some $\eta=\eta(\eps,B,N)>0$. Provided $r<\sqrt{\delta /|K|}$ we can apply the (scaled version) of Proposition \ref{prop:boundary implies bilip} to obtain a set $U\subset B_{r'/16}(x')\cap \partial \X\subset B_{r}(x)\cap \partial \X$ that is $(1+\eps)$-biLipschitz to a subset of $\rr^{N-1}$ and such that 
    $$\hau^{N-1}(U)\ge \frac12 \hau^{N-1}(B_{r'/16}(x)\cap \partial X)\ge \frac{C_B^{-1}}2 (r'/16)^{N-1} \ge\frac{C_B^{-1}}2  (\eta r/16)^{N-1},$$
    having used that $\sfd(x',B)<1.$ This proves that $\partial \X$ has locally BPBI of $\rr^{N-1}$ and thus concludes the proof. 
\end{proof}

\subsection{BWGL and quantitative differentiation}\label{s:quant-diff}

Theorem \ref{thm:rcd ur bounded} (resp.\ Theorem \ref{thm:UR boundary}) tells us that that bounded Ahlfors regular $\rcd$ spaces (resp.\ Ahlfors regular boundaries of bounded non-collapsed $\rcd$ spaces) are UR. By applying the results in \cite{BHS23} we deduce that these spaces also satisfy the \textit{Bilateral Weak Geometric Lemma} (BWGL) (see \cite[Definition 3.1.5]{BHS23}), which roughly states that the space is uniformly approximated by Banach spaces at most scales and locations.  However it is worth noting that we can deduce BWGL directly. Actually in this way we obtain a slightly stronger version of BWGL, stated below, where the comparison is made only with the Euclidean $\rr^n$ (rather than Banach spaces).

 \begin{prop}\label{prop:BWGL}
Let $(X,d,\mm)$ be a bounded Ahlfors $k$-regular m.m.s., $k \in \nn,$ being  either an  $\RCD(K,N)$ space with $N<\infty$ or the boundary of a non-collapsed $\RCD(K,k+1)$ space (endowed with the restriction distance and  $\hau^k$-measure). Then, for all $\eps > 0$ there exists a constant $C(\eps) > 1$ such that
	\begin{align}\label{e:BWGL-carleson}
		\int_0^R \mathcal H^k( \{x \in B_R^X(x_0) : \sfd_{GH}(B_r^X(x),B_r^{\rr^{k}}(0)) > \eps \}) \, \frac{dr}{r} \leq C(\eps) R^k
	\end{align}
	holds for all $x_0 \in X$ and $0 < R < \emph{diam}(X).$ 
\end{prop}

\begin{proof}
Let $(X,d,\mm)$ be a bounded Ahlfors $k$-regular m.m.s., with $k \in \nn$ and Ahlfors regularity constant $C$, and being  an  $\RCD(K,N)$ space with $N<\infty$ (resp.\ the boundary of a non-collapsed $\RCD(K,k+1)$ ambient space $(Y, \sfd_Y,\hau^{k+1})$).
    Let $\eps > 0$ and let $\delta=\delta(N,\eps,C)$  (resp.\ $\delta=\delta(\eps,k+1)$)  be the constant  given by Proposition \ref{prop:flat implies bilip} (resp.\ by Proposition \ref{prop:boundary implies bilip}). We claim, for all $x_0 \in X$ and $0 < R < \min\{\diam(X),\sqrt{\delta/|K|}\},$ that
        \begin{equation}\label{e:BWGL-delta}
        \int_0^R \mathcal H^k\big(\{y \in B_{R}(x_0) \ : \sfd_{GH}(B_s(x),B_s^{\rr^k}(0))>\eps s\}\big) \frac{\d s}{s}\le \delta^{-1}R^k.
        \end{equation}
    Since $X$ is bounded and Ahlfors $k$-regular, this proves \eqref{e:BWGL-carleson} for all $x_0 \in X$ and $0 < R < \diam(X)$ at the cost of increasing the constant on the right-hand side depending on $K$, $\diam(\X)$ and $C$. 
    
    Fix now $x_0 \in X$ and $0 < R < \min\{\diam(X),\sqrt{\delta/|K|}\}$. We may assume that $R = 1$ and $K = -\delta.$ Otherwise we can just consider the rescaled space $(\X,R^{-1}\sfd,\mea)$ (resp.\ the rescaled ambient space $(Y, R^{-1}\sfd_Y,\hau^{k+1})$)  which is an $\rcd(-\delta,N)$ space (resp.\ an $\rcd(-\delta,k+1)$ space), thanks to the assumption $R<\sqrt{\delta/|K|}$ (recall \eqref{eq:rcd scaling}). Using (the scaled versions of) Proposition \ref{prop:jonas lemma} and Proposition \ref{prop:flat implies bilip} (resp.\ Proposition \ref{prop:boundary jonas lemma} and Proposition \ref{prop:boundary implies bilip}), for every ball $B_r(x)$ with  $x\in B_1(x_0)$ and $r<1$, there exists a  $U\subset B_r(x)$ satisfying  \eqref{e:U bilip} (resp.\ \eqref{eq:bwgl boundary}) and $\mm(U)\ge c \mm(B_r(x))$, where $c>0$ is a constant independent of $r$ and $x$. We obtain \eqref{e:BWGL-delta} from this by a simple application of the John-Nirenberg lemma (see \cite[Lemma IV.1.2]{david-semmes93}). The lemma in \cite{david-semmes93} is stated for Euclidean space, but the argument translates verbatim to general Ahlfors regular metric measure spaces. Applying the John-Nirenberg lemma requires a dyadic lattice to be in place, and a form of quasi monotonicity, with respect to balls, of the local Gromov-Hausdorff distance in \eqref{e:BWGL-delta}. Both are available, see \cite[Section 2.6]{BHS23} and \cite[eq. (2.2.3)]{BHS23}, respectively. For the sake of brevity we omit the detailed proof, leaving the details to the reader.
\end{proof}

On these same spaces, as a corollary to Theorem \ref{thm:rcd ur bounded} and \cite[Theorem 4.1.1]{BHS23}, we obtain a result on quantitative differentiability of Lipschitz functions. To state this we need some notation. 	
\begin{definition}
	Let $(X,\sfd)$ a metric space, $L \geq 1$ and $f \colon X \to \R.$ For $x \in X$, $0 < r < \diam(X)$ and $u \colon B_r(x) \to B^{\R^k}_r(0)$ define 
	\begin{align}
		\zeta(x,r,u) &= \frac{1}{r} \sup_{y,z \in B_r(x)} \left| \sfd(y,z) - | u(y) - u(z) | \right| \\
		\eta(x,r,u) &= \frac{1}{r} \sup_{y \in B^{\R^k}_r(0)} \dist(y, u(B_r(x)) ) \\
		\Omega^L_f(x,r,u) &= \frac{1}{r} \inf_A \sup_{y \in B_r(x)} | f(y) - A(u(y)) |, 
	\end{align}
	where the infimum in the final equation is taken over all affine maps $A \colon \R^k \to \R$ with $\text{Lip}(A) \leq L$. Finally, let 
	\begin{align}
		\gamma^L_f(x,r) = \inf_{u} \left[ \zeta(x,r,u) + \eta(x,r,u) + \Omega^L_f(x,r,u) \right], 
	\end{align}
	where the infimum is taken over all maps $u \colon B_r(x) \to B^{\R^k}_r(0).$ 
\end{definition}

\begin{remark}\label{r:approx}
	If $\gamma_f^L(x,r) < \eps,$ then there exist a map $u \colon B_r(x) \to B^{\R^k}_r(0)$ which is an $\eps r$-GH isometry and a $L$-Lipschitz affine map on $u(B_r(x))$ well-approximates $f$ up to an error $\eps r.$ Since $u$ is an $\eps r$-GH isometry, one may like to think of $A$ as being approximately an ``affine map on $B_r(x)$''.
\end{remark}

\begin{remark}\label{r:gamma-tilde}
	The above definition corresponds to \cite[Definitions 3.1.3 and 3.1.6]{BHS23}. There, one considers maps $u$ taking values in $B^{(\R^k,\|\cdot\|)}_r(0)$ for some norm $\|\cdot\|$ and a further infimum is taken over all norms. This comes from the fact that arbitrary UR metric spaces may possess non-Euclidean tangents. Denoting by $\tilde{\gamma}^L_f(x,r)$ the coefficient from \cite{BHS23}, then $\tilde{\gamma}^L_f(x,r) < \eps$ means there exists a norm $\|\cdot\|$ on $\R^k,$ a map $u \colon B_r(x) \to B^{(\R^k,\|\cdot\|)}_r(0)$ which is an $\eps r$-GH isometry and a $L$-Lipschitz affine map $A \colon (\R^k,\|\cdot\|) \to \R$ which approximates $f$ up to an $\eps r$ error. 
\end{remark}

The following corollary states that the type of approximation described in Remark \ref{r:approx} occurs at \textit{most} locations and scales. 

\begin{cor}\label{c:quant-diff}
Let $(X,d,\mm)$ be a bounded Ahlfors $k$-regular m.m.s., $k \in \nn,$ being  either an  $\RCD(K,N)$ space with $N<\infty$ or the boundary of a non-collapsed $\RCD(K,k+1)$ space (endowed with the restriction distance and  $\hau^k$-measure). Then, there exists $L \geq 1$ and for each $\eps > 0$ a constant $C(\eps) > 1$ such that if $f \colon X \to \R$ is $1$-Lipschitz then 
	\begin{align}\label{e:gamma-carleson}
		\int_0^R \mathcal H^k( \{x \in B_R(x_0) : \gamma_f^L(x,r) > \eps \}) \, \frac{dr}{r} \leq C(\eps) R^k 
	\end{align}
	for all $x_0 \in X$ and $0 < R < \emph{diam}(X).$ 
\end{cor}

\begin{proof}
  
    Let $\tilde{\gamma}_f^L$ denote the coefficients from \cite[Definition 3.1.6]{BHS23}. By Theorem \ref{thm:rcd ur bounded}, $X$ is uniformly $k$-rectifiable, so, by \cite[Theorem 4.1.1]{BHS23}, there exists $L' \geq 1$ and for each $\delta > 0$ a constant $C(\delta) > 1$ such that \eqref{e:gamma-carleson} holds with the coefficients $\tilde{\gamma}_f^{L'}$ and constant $C(\delta)$. This and Proposition \ref{prop:BWGL} imply \eqref{e:gamma-carleson} for the coefficient $\gamma_f^L$ once we establish the following claim: Let $L = 2L'.$ For each $\eps > 0$ there exists $\delta > 0$ such that if
	\begin{align}
		\mathscr{G}_\delta \coloneqq \{ (x,r) \in X \times (0,\diam(X)) : \sfd_{GH}(B_r(x),B^{\R^k}_r(0)) \leq \delta r \mbox{ and } \tilde{\gamma}^{L'}_f(x,r) \leq \delta   \}
	\end{align}
	then 
	\begin{align}
		\mathscr{G}_\delta \subseteq \{(x,r) \in X  \times (0,\diam(X)) : \gamma_f^L(x,r) \leq \eps \}. 
	\end{align}
	We now prove the claim. Let $\delta,\eps > 0$ and let $(x,r) \in \mathscr{G}_\delta.$ By Remark \ref{r:gamma-tilde}, there exists a norm $\|\cdot\|$, a $\delta r$-GH isometry $\tilde{u} \colon B_r(x) \to B^{(\R^k,\|\cdot\|)}_r(0)$ and an affine function $\tilde{A} \colon \R^k \to \R$ such that $\text{Lip}(\tilde{A}) \leq L'$ (as a map from $(\R^k,\|\cdot\|)$ to $\R$) and 
\begin{align}\label{e:opt-A}
		\sup_{y \in B_r(x)} | f(y) - \tilde{A}(\tilde{u}(y)) |\leq r \delta  . 
	\end{align} 
	The first part, combined with the definition of $\mathscr{G}_\delta$, implies $\sfd_{GH}(B_r^{\R^k}(0),B_r^{(\R^k,\|\cdot\|)}) < 2\delta r$, hence, there exists a $2\delta r$-GH isometry $\varphi \colon B_r^{(\R^k,\|\cdot\|)}(0) \to B^{\R^k}_r(0)$. By \cite[Lemma 2.3.14]{BHS23}, for $\delta$ small enough depending on $\eps$ and $L$, there exists a $(1+\eps)$-bi-Lipschitz affine map $T \colon B_r^{(\R^k,\|\cdot\|)}(0) \to B^{\R^k}_r(0)$ such that 
	\begin{align}\label{e:lin-T}
		\sup_{y \in B^{(\R^k,\|\cdot\|)}_r(0)} | T(y)-\varphi(y)  | \leq \frac{\eps r}{4L}. 
	\end{align}
	Define 
	\begin{align}
		u \coloneqq \varphi \circ \tilde{u} \colon B_r(x) \to B^{\R^k}_r(0) \mbox{ and } A \coloneqq \tilde{A} \circ T^{-1} \colon \R^k \to \R. 
	\end{align}
	For $\delta$ small enough, since $\tilde{u},\varphi$ are $2\delta r$-GH isometries, it follows that $u$ is an $\eps r/2$-GH isometry. In particular, we have 
	\begin{align}\label{e:zeta-eta}
		r\zeta(x,r,u) + r\eta(x,r,u) \leq \eps r/2.
	\end{align}
	Since $0 < \eps < 1$ and using the fact that $\tilde{A}$ is $L'$-Lipschitz and $T^{-1}$ is $(1+\eps)$-Lipschitz, we have $\text{Lip}(A) \leq 2L' = L$ (viewed as a map from $\R^k$ to $\R$). Combining this with \eqref{e:opt-A} and \eqref{e:lin-T} gives 
	\begin{align}\label{e:omega}
		r \Omega_f^L(x,r,u) &\leq \sup_{y \in B_r(x)}	| f(y) - A(u(y)) | \\
		&\leq\sup_{y \in B_r(x)} \left[ | f(y) - \tilde{A} ( \tilde{u}(y) ) | + | A \circ T \circ \tilde{u}(y) - A \circ \varphi \circ \tilde{u}(y) | \right] \\
		&\le \delta r + L \sup_{y \in B_r^{\R^k}(0)} | T(\tilde{u}(y)) - \varphi(\tilde{u}(y)) |\leq \eps r/2,
	\end{align}
 provided $\delta \le \eps/4.$
	Equations \eqref{e:zeta-eta} and \eqref{e:omega} finish the proof of the claim. 
\end{proof}

\def\cprime{$'$}

\end{document}